\documentclass[11pt]{amsart}

\usepackage{amscd}
\usepackage{xypic}  
\usepackage{amssymb}
\usepackage{amsthm}
\usepackage{epsfig}
\usepackage{paralist}

\usepackage{tikz}

\usepackage[all,cmtip]{xy}
\usepackage{amsmath}
\usepackage{color}

\newtheorem{thm}{Theorem}[section]

\newtheorem{lemma}[thm]{Lemma}
\newtheorem{prop}[thm]{Proposition}
\newtheorem{proposition}[thm]{Proposition}

\renewcommand{\proofname}{Proof}

\newtheorem{corollary}[thm]{Corollary}
\theoremstyle{definition}

\newtheorem{remark}[thm]{Remark}

  \newtheorem{definition-remark}[thm]{Definition-Remark}
 \newtheorem{example}[thm]{Example}

\def\ker{\operatorname{ker}}

\def\cork{\operatorname{cork}}

\def\min{\operatorname{min}}

\def\c1{\operatorname{c_1}}
\def\c2{\operatorname{c_2}}
\def\Cliff{\operatorname{Cliff}}

\def\PP{{\mathbb P}}

\def\N{{\mathcal N}}
\def\O{{\mathcal O}}
\def\I{{\mathcal J}}

\def\E{{\mathcal E}}
\def\T{{\mathcal T}}

\def\F{{\mathcal F}}

\def\FF{{\mathbb F}}

\def\c{\mathfrak{c}}

\def\R{\mathcal{R}}

\def\x{\times}                   
\def\cong{\simeq}

\def\+{\oplus}               
\def\*{\otimes}                  

\def\Shom{\operatorname{ \mathfrak{h}\mathfrak{o}\mathfrak{m} }}

\def\Pic{\operatorname{Pic}}

\begin{document}

\title[Global sections of twisted normal bundles]{Global sections of twisted normal bundles of $K3$ surfaces and their hyperplane sections} 

\author{Andreas Leopold Knutsen}
\address{Andreas Leopold Knutsen, Department of Mathematics, University of Bergen, Postboks 7800,
5020 Bergen, Norway}
\email{andreas.knutsen@math.uib.no}



 \begin{abstract}  
 Let $S \subset \PP^g$ be a smooth $K3$ surface of degree $2g-2$, $g \geq 3$. We classify all the cases for which $h^0(\N_{S/\PP^g}(-2)) \neq 0$ and the cases for which $h^0(\N_{S/\PP^g}(-2)) < h^0(\N_{C/\PP^{g-1}}(-2))$ for $C \subset \PP^{g-1}$ a general canonical curve section of $S$.
 \end{abstract}

\maketitle

\section{Introduction}

The spaces of global sections of twists of normal bundles  of an embedded variety $X \subset \PP^n$ in projective space naturally occur in many ways, for instance in the deformation theory of the cone over $X$, cf. \cite{Pi}. More specifically, the spaces $H^0(\N_{X/\PP^n}(-k))$ for $k=1,2$ are related to extendability properties of $X$, as we now briefly recall. 

An {\it $r$-step extension} of a smooth variety $X \subset \PP^n$ is a projective variety $W \subset \PP^{n+r}$ so that $X$ is the transversal intersection of $W$ with a $\PP^n \subset \PP^{n+r}$. If $W$ is not a cone, then the extension is called {\it nontrivial}, and $X$ is called {\it $r$-extendable}. A famous theorem of Zak-Lvovski \cite{Zak,Lvov} states that if $X$ is not a quadric, and $h^0(\N_{X/\PP^n}(-1)) < \min\{n+1+r,2n+1\}$, then $X$ is not $r$-extendable. Quite remarkably, a converse of the theorem of Zak-Lvovski was recently obtained in \cite[Thms. 2.1 and 2.19]{cds} in the case of $X$ a canonical curve or a $K3$ surface, to the effect that $h^0(\N_{X/\PP^n}(-1)) \geq n+1+r$ is a {\it sufficient} condition for $r$-extendability, provided that the curve (respectively, any smooth hyperplane section of the surface) has genus at least $11$ and Clifford index at least $3$. 

Whereas $H^0(\N_{X/\PP^n}(-1))$ is connected with the {\it existence} of nontrivial $r$-extensions $X \subset W$, the space $H^0(\N_{X/\PP^n}(-2))$ is connected with {\it uniqueness}, as was proved by Wahl \cite[Thm. 1.9 and Thm. 2.8]{W-CY}: If 
$H^0(\N_{X/\PP^n}(-2))=0$, then $W$ is uniquely determined by its Kodaira-Spencer map (see \cite[(1.4)]{W-CY} for its definition), and if in addition the Kodaira-Spencer map is an isomorphism, then $W$ is {\it universal}, meaning that every extension of $X$ is equivalent to a (possibly trivial) cone over a unique subextension (see also \cite[\S 4]{cds}).

Much attention has been devoted to canonical curves and $K3$ surfaces. We refer to the recent works \cite{abs,cds}, and recall, as another instance, that considerations as above led to the proof that a curve of genus $g \geq 11$, $g \neq 12$ lying on a $K3$ surface, is generically contained in at most one such surface \cite{clm}.

In this paper we will concentrate on the computations of the spaces of global sections $H^0(\N_{S/\PP^g}(-k))$ for $k \geq 2$ in the case of $K3$ surfaces $S \subset \PP^g$ of degree $2g-2$. (The case of {\it general} $K3$ surfaces was treated in \cite{clm-class}.) Their hyperplane sections are canonical curves
$C \subset \PP^{g-1}$ of genus $g$ and we start by recalling the known results in this case, before stating our main result (Proposition \ref{prop:main2} below). The cohomology groups in question are related to the well-known gaussian maps
\[ \Phi_{\omega_C,\omega_C^{\*l}}: \ker \mu_{\omega_C,\omega_C^{\*l}} \longrightarrow H^0(\omega_C^{\*(l+1)}), \]
where $\ker \mu_{\omega_C,\omega_C^{\*l}}$ is the kernel of the multiplication map
\[ \mu_{\omega_C, \omega_C^{\*l}}: H^0(\omega_C) \* H^0(\omega_C^{\*l}) \to 
H^0(\omega_C^{\*(l+1)}),\] 
and the map $ \Phi_{\omega_C,\omega_C^{\*l}}$ is (essentially) defined by sending $\sigma \* \tau$ to $d\sigma \* \tau-\sigma \* d\tau$.  
We have (cf. \cite{W1} or \cite[Prop. 1.2]{CM1}):
\begin{eqnarray}
  \label{eq:gauss1}
  h^0(\N_{C/\PP^{g-1}}(-1)) & = & g + \cork\Phi_{\omega_C,\omega_C}, \\
  \label{eq:gauss2} h^0(\N_{C/\PP^{g-1}}(-j)) & = & 
\cork\Phi_{\omega_C,\omega_C^{\*j}} \; \; \mbox{for} \; \; j \geq 2. 
\end{eqnarray}

The study of $h^0(\N_{C/\PP^{g-1}}(-1))$, or equivalently, of the corank of the gaussian map $\Phi_{\omega_C,\omega_C}$, is a tricky question and a history of its own. We refer for instance to the works \cite{W2,W,W3,chm,CM1,CM2} and the very recent works on $K3$ surfaces \cite{abs,cds}. It is still an open question to determine the possible values of this for all curves, although the value is known for general curves and for a general curve of any fixed gonality. We will in the following concentrate on the dimensions $h^0(\N_{C/\PP^{g-1}}(-j))$ for $j \geq 2$ and restrict our attention to the cases $g \geq 5$, as otherwise the canonical model is a complete intersection and the cohomology groups can be easily 
calculated.

Recall the following well-known fact, cf. \cite[\S 2]{W2}, \cite[Lemma 2.7(ii)]{KL-gauss} or \cite[Lemma 3.5]{cds}:

\begin{lemma} \label{lemma:well-known}
  Let $X \subset \PP^n$ be a locally complete intersection variety such that the homogeneous ideal of $X$ is generated by quadrics and the first syzygy module is generated by linear syzygies. Then $h^0(\N_{X/\PP^n}(-k))=0$ for all $k \geq 2$. 
\end{lemma}

An immediate consequence of this, together with Petri's theorem and  results on syzygies of tetragonal curves by Schreyer \cite{sc} and Voisin \cite{Vo}, is the following well-known fact that can also be deduced from \cite[Thm. 2]{BEL} and \eqref{eq:gauss2}:

\begin{corollary} \label{cor:well-knowncancurve}
  Let $C \subset \PP^{g-1}$ be a canonically embedded (nonhyperelliptic) curve of genus $g \geq 3$. If $\Cliff C \geq 3$, then $h^0(\N_{C/\PP^{g-1}}(-k))=0$ for all $k \geq 2$.
\end{corollary}

Secondly, as the gaussian maps $\Phi_{\omega_C,\omega_C^{\*l}}$ are well-known to be surjective for $l \geq 3$ and $g \geq 5$, cf., e.g.,  \cite[Cor. 2.10 and Prop. 2.11]{KL-gauss} for a proof, we have by \eqref{eq:gauss2} that
\begin{equation} \label{eq:Ckgeq3}
h^0(\N_{C/\PP^{g-1}}(-k))=0 \; \; \mbox{for all} \; \;k \geq 3 \; \; \mbox{(when $g \geq 5$)}.
\end{equation}

The cases left are the cases $k=2$ for curves of Clifford indices one and two, that is, trigonal and tetragonal curves, as well as curves isomorphic to smooth plane quintics or sextics. The possible values have been computed  in various works, cf. \cite{DS,Te,BEL,CM2}, although a complete statement seems to be missing in the literature. We give all possible values of $h^0(\N_{C/\PP^{g-1}}(-2))$, for $k \geq 2$, for such curves, in Propositions \ref{prop:c=1}, \ref{prop:c=1SPQ},
\ref{prop:c=2} and \ref{prop:c=2SPS} and remark that all the possible values actually do occur (more precisely, for curves on $K3$ surfaces, cf. Remark \ref{rem:SMP}, Example \ref{ex:genere7} and Proposition \ref{prop:main2}). In particular, we see that 
\begin{equation} \label{eq:Ck2}
h^0(\N_{C/\PP^{g-1}}(-2))=0 \; \; \mbox{if $g \geq 11$ and $C$ is not bielliptic}.
\end{equation}

Since hyperplane sections of $K3$ surfaces are canonical curves, an immediate consequence of Lemma \ref{lemma:well-known}, 
together with Green's hyperplane section theorem on syzygies \cite[Thm. 3.b.7]{gr}, is the following fact, also well-known.
Recall that 
the Clifford index is constant among all smooth curves in a complete linear system on a $K3$ surface, cf. \cite{gl}. 

\begin{corollary} \label{cor:well-knownK3}
  Let $S \subset \PP^g$ be a (possibly singular) projective model of a  $K3$ surface of degree $2g-2$. Let $c$ be the Clifford index of the smooth hyperplane sections of $S$. 

If $c \geq 3$, then $h^0(\N_{S/\PP^g}(-k))=0$ for all $k \geq 2$.
\end{corollary}

Here, by a projective model of a $K3$ surface of degree $2g-2$ in $\PP^g$, we mean the image under a 
birational morphism $\varphi_H$ defined by a complete linear system $|H|$ on a smooth $K3$ surface, where $H \in \Pic S$. As is well-known, $H^2=2(g-1)$, the general members of such a linear system are smooth, nonhyperelliptic curves of genus $g$, and the morphism  $\varphi_H$ is an isomorphism except for the possible contraction of (chains of) smooth rational curves. The image surface is normal, with at most isolated, rational double points as singularities,  cf. \cite{SD}. 

It is very easy to see, cf. Lemma \ref{lemma:primo}, that
\[ h^0(\N_{S/\PP^g}(-k)) \leq \sum_{j=0}^{\infty} h^0(\N_{C/\PP^{g-1}}(-k-j)).\]
Hence immediate consequences of \eqref{eq:Ckgeq3} are
\begin{equation} \label{eq:Skgeq3}
h^0(\N_{S/\PP^{g}}(-k))=0 \; \; \mbox{for all} \; \;k \geq 3 \; \; \mbox{(when $g \geq 5$)}
\end{equation}
and
\begin{equation} \label{eq:SCk2}
h^0(\N_{S/\PP^{g}}(-2)) \leq  h^0(\N_{C/\PP^{g-1}}(-2)) \; \; \mbox{(when $g \geq 5$)}
\end{equation}
(again for a possibly singular projective model $S \subset \PP^g$ of a  $K3$ surface).

Moreover, as there are no bielliptic curves of genus $g \geq 11$ on a $K3$ surface by a result of Reid's \cite[Cor. 2]{Reid}, an immediate consequence of
 \eqref{eq:Ck2} is that
\begin{equation} \label{eq:Sk2}
h^0(\N_{S/\PP^{g}}(-2))=0 \; \; \mbox{if $g \geq 11$}.
\end{equation}
(This was already implicitly contained in \cite[Pf. of Thm. 3.2]{clm-class}.)
The main result of this paper gives an explicit classification of all {\it smooth} projective models of $K3$ surfaces such that $h^0(\N_{S/\PP^{g}}(-2)) \neq 0$:

\begin{proposition} \label{prop:main2}
  Let $S \subset \PP^g$ be a smooth  $K3$ surface of degree $2g-2$, with $g \geq 5$. 

 If $g=5$, then $h^0(\N_{S/\PP^5}(-2))=3$, and if $g=6$, then $h^0(\N_{S/\PP^6}(-2))=1$.

 If $g \geq 7$, then $h^0(\N_{S/\PP^g}(-2))=0$ except for the following cases, where  $h^0(\N_{S/\PP^g}(-2))=1$:
\begin{itemize}
\item[(I)] $g=7$  and 
$\O_S(1) \sim 3E+\Gamma_1+\Gamma_2+\Gamma_3$, where $|E|$ is an elliptic 
pencil of degree three on $S$ and $\Gamma_1,\Gamma_2,\Gamma_3$ are disjoint lines (with $\Gamma_i \cdot E=1$, $i=1,2,3$). 
\item[(II)] $g=7$  and there are  
three elliptic pencils $|E_i|$ on $S$, $i=1,2,3$, such that $E_i \cdot E_j=2$ for $i \neq j$ and 
$\O_S(1) \sim E_1+E_2+E_3$. 
\item[(III)] $g=7$  and there is a globally generated line bundle $D$ on $S$ satisfying $D^2=2$ and $D \cdot H=6$. 
\item[(IV)] $g=8$ and there is a globally generated line bundle $D$ on $S$ satisfying $D^2=2$ and $D \cdot H=6$. 
\item[(V)] $g=9$ and $H \sim 2D$ with $D^2=4$. 
\item[(VI)] $g=9$ and $H \sim 3E+2\Delta$, where $|E|$ is an elliptic pencil and $\Delta$ is an effective divisor such that $\Delta^2=-2$ and $\Delta \cdot E=2$. 
\item[(VII)] $g=10$ and $H \sim 3D$ with $D^2=2$.
\end{itemize}
\end{proposition}

We remark that the statement for $g=5$ and $6$  is of course well-known for {\it general} $K3$ surfaces, more precisely, with the Clifford index of all smooth hyperplane sections $c=2$, as they are complete intersections of three quadrics for $g=5$ and 
 quadratic  sections of a (possibly singular) quintic Del Pezzo threefold  in $\PP^6$. It is however new for $K3$ surfaces with $c=1$, at least as far as we know. We also remark that $c=1$ in (I) and $c=2$ in (II)-(VII).

A general surface $S$ in each of the cases is of the following form, and conversely all surfaces below belong to the cases listed in the proposition (in particular, all cases do occur): 
\begin{itemize}
\item[(I)] $S$ 
lies in a three-dimensional rational normal scroll $T$ of type 
$(3,1,1)$ in $\PP^7$ as a 
a divisor in $|\left(\O_T(1)(-\F)\right)^{\*3}|$, where $\F$ is the class of the ruling of $T$.
\item[(II)] $S$ is a 
 quadratic section of the sextic Del Pezzo threefold $T \cong \PP^1 \x \PP^1 \x \PP^1$ in its Segre embedding in $\PP^7$.
\item[(III)] $S$ is a quadratic section of the sextic Del Pezzo threefold $W$ in $\PP^7$ that is a divisor of bidegree $(1,1)$ in $\PP^2 \x \PP^2$. 
\item[(IV)] $S$ is a quadratic section of a blow up of $\PP^3$ at a point embedded in $\PP^8$ by the linear system of quadrics through the point (a septic Del Pezzo threefold). 
\item[(V)] $S$ is the $2$-Veronese embedding of a quartic in $\PP^3$, and thus a quadratic section of the $2$-Veronese embedding of $\PP^3$ in $\PP^{9}$.
\item[(VI)] $S$ is a 
 quadratic section of the cone over the anticanonical embedding of the Hirzebruch surface $\FF_1$ in $\PP^8$.    
\item[(VII)] $S$ is a quadratic section of the cone  over the Veronese surface in $\PP^9$.
\end{itemize}

Finally, in Proposition \ref{prop:main3}, we classify the cases for which the strict inequality 
 $h^0(\N_{S/\PP^{g}}(-2)) < h^0(\N_{C/\PP^{g-1}}(-2))$ holds, again in the case of {\it smooth} projective models.

\vspace{0.4cm} {\it Acknowledgements.} This paper grew out of my interest in the recent paper \cite{cds}. I thank C.~Ciliberto, T.~Dedieu and E.~Sernesi for the many conversations on this topic, and in particular, C.~Ciliberto for encouraging me to write down these results. I also thank A.~F.~Lopez for useful conversations and for indicating several references, as well as the referee for detecting various misprints.  

I have been partially supported by grant n. 261756 of
the Research Council of Norway and by the Bergen Research Foundation.

\section{Some useful results}

The following result was already mentioned in the introduction:

\begin{lemma} \label{lemma:primo}
  Let $X \subset \PP^n$ be a local complete intersection  surface with isolated singularities. Then, for any smooth hyperplane section $C \subset X$, we have
\[
h^0(\N_{X/\PP^n}(-k)) \leq \sum_{j=0}^{\infty} h^0(\N_{C/\PP^{n-1}}(-k-j)).
\]
\end{lemma}

\begin{proof}
  By assumption, $X$ has smooth hyperplane sections. For any such $C$, we have
$\N_{C/\PP^{n-1}} \cong \N_{X/\PP^n} \* \O_C$.
The exact sequences
\[ 
0 \longrightarrow \N_{X/\PP^n} (-k-1-j) \longrightarrow  \N_{X/\PP^n} (-k-j) \longrightarrow  
\N_{X/\PP^n}(-k-j) \* \O_C \longrightarrow  0
\]
thus yield the desired result.
\end{proof}

We will need the following strengthening of Lemma \ref{lemma:well-known}, proved in \cite[Lemma 2.7]{KL-gauss}:

\begin{lemma} \label{lemma:well-known2}
  Let $Y \subset \PP^n$ be an integral subvariety such that the homogeneous ideal of $Y$ is generated by quadrics and the first syzygy module is generated by linear syzygies (e.g., $Y$ satisfies property $N_2$). Let $X \subset Y$ be a smooth irreducible nondegenerate subvariety.

Then $h^0(\Shom_{\O_{\PP^n}}(\I_{Y/\PP^n},\O_X)(-2))=0$.
\end{lemma}

We will also make use of the following simple observation:

\begin{lemma} \label{lemma:trick}
  Let $S \subset \PP^g$ be a smooth $K3$ surface of degree $2g-2$. Then 
$h^0(\N_{S/\PP^g}(-2))=h^1(\T_S(-2))$.
\end{lemma}

\begin{proof}
  The Euler sequence twisted by $\O_S(-2)$ is
\[
\xymatrix{ 0 \ar[r] & \O_S(-2) \ar[r] & H^0(\O_S(1))^{\vee} \* \O_S(-1) \ar[r] & \T_{\PP^g}|_S(-2)  \ar[r] & 0.}\]
The map on cohomology $H^2(\O_S(-2)) \to H^0(\O_S(1))^{\vee} \* H^2(\O_S(-1))$ is the dual of the multiplication map $H^0(\O_S(1)) \* H^0(\O_S(1)) \to H^0(\O_S(2))$, which is surjective by \cite[Thm. 6.1(ii)]{SD}. Thus,  $h^i(\T_{\PP^g}|_S(-2))=0$ for $i=0,1$. The desired conclusion now follows from the exact sequence
\[
\xymatrix{ 0 \ar[r] & \T_S(-2) \ar[r] & \T_{\PP^g}|_S(-2)  \ar[r] & \N_{S/\PP^g}(-2) \ar[r] & 0.}
\]
\end{proof}

\section{The case of Clifford index one} \label{sec:c=1}

It is well-known that smooth curves of Clifford index one are either trigonal or isomorphic to smooth plane quintics. The next two results give all possible values of $h^0(\N_{C/\PP^{g-1}}(-2))$, or equivalently, $\cork \Phi_{\omega_C,\omega_C^{\*2}}$, by \eqref{eq:gauss2}, for such curves. 

The following result was proved in \cite[\S 3.8]{DS} in terms of {\it Maroni invariants} of the rational normal scroll defined by the $g^1_3$, cf. \cite{sc}. The result is apparently also contained in an unpublished preprint of Tendian \cite{DS}. We formulate the result in a slightly different way and prove it using \cite{KL-gauss}, which in principle adopts the same idea of proof as \cite{DS}. Note that the case of {\it general} trigonal curves was proved in \cite[Thm. 2.8]{CM2}.

\begin{prop} \label{prop:c=1}
  Let $C$ be a smooth trigonal curve of genus $g \geq 5$ and denote by $A$ its unique line bundle of type $g^1_3$.
  \begin{itemize}
  \item[(i)] If $g \geq 11$, then $h^0(\N_{C/\PP^{g-1}}(-2))=0$.
  \item[(ii)] If $g=10$, then $h^0(\N_{C/\PP^9}(-2))=0$, unless $\omega_C \cong 6A$, in which case one has $h^0(\N_{C/\PP^{9}}(-2))=1$.
\item[(iii)] If $g=8$ or $9$, then $h^0(\N_{C/\PP^{g-1}}(-2))=h^0(\omega_C-(g-4)A) \leq 1$.
\item[(iv)] If $g=7$, then $h^0(\N_{C/\PP^{6}}(-2))=1$, unless $\omega_C \cong 4A$, in which case one has $h^0(\N_{C/\PP^{6}}(-2))=2$.
\item[(v)] If $g=6$, then $h^0(\N_{C/\PP^{5}}(-2))=2$.
\item[(vi)] If $g=5$, then $h^0(\N_{C/\PP^{4}}(-2))=3$.
  \end{itemize}
\end{prop}

\begin{proof}
In the canonical embedding $C \subset \PP^{g-1}$, the members of the $g^1_3$ on $C$ are collinear, and the lines sweep out a rational normal surface $Y \subset \PP^{g-1}$ containing $C$, cf. \cite[\S 4 and 6.1]{sc}. Since the $g^1_3$ is base point free, the curve $C$ does not intersect the possibly empty singular locus of $Y$ and $C \in |\O_Y(3)(-(g-4)\R)|$, where $\R$ is the class of the ruling of $Y$, cf. \cite[\S~6.1]{sc}. 

We have the twisted normal bundle sequence (recalling that  $Y$ is smooth along $C$): 
\[ 
\xymatrix{
0 \longrightarrow  \N_{C/Y}(-2) \ar@{=}[d]^{\wr} \ar[r] &  \N_{C/\PP^{g-1}}(-2) \ar[r] &  \N_{Y/\PP^{g-1}}|_C(-2) \ar[r] & 0. \\
\O_C(1)(-(g-4)\R) & & &
}\]
Since $Y$ satisfies property $N_2$ (as any of its smooth hyperplane sections does), Lemma \ref{lemma:well-known2} yields 
\[ h^0( \N_{C/\PP^{g-1}}(-2))=h^0(\O_C(1)(-(g-4)\R)=
h^0(\omega_C-(g-4)A).\]
 Because $\deg (\omega_C-(g-4)A)=10-g$, items (i)-(iv) easily follow. 
(In item (iv) one uses that $h^0(3A) \geq 4$, to conclude by Riemann-Roch that
$h^0(\omega_C-3A) \geq 1$.)

If $g=6$, we get $h^0(\omega_C-2A) \leq 2$, since otherwise $C$ would carry a $g^2_4$ and thus be hyperelliptic, a contradiction. At the same time, $h^0(2A) \geq 3$, so that $h^0(\omega_C-2A) \geq 2$ by Riemann-Roch. Hence (v) follows. If
$g=5$, then $h^0(\omega_C-A)=h^1(A)=3$, and (vi) follows. 
\end{proof}

\begin{prop} \label{prop:c=1SPQ}
  If $C$ is isomorphic to a smooth plane quintic (whence of genus $6$), then $h^0( \N_{C/\PP^{g-1}}(-2)) =3$.
\end{prop}

\begin{proof}
  This is \cite[Thm. 2.3]{CM2}. Alternatively, it follows from \cite[Prop. 2.9(d)]{KL-gauss}. 
\end{proof}

\begin{remark} \label{rem:SMP}
  It is well known, cf. \cite{SD}, that a curve $C$ on a $K3$ surface is isomorphic to a smooth plane quintic if and only if $C \sim 2B+\Gamma$, where $B$ is a smooth genus $2$ curve, $\Gamma$ is a smooth rational curve and $B \cdot \Gamma=1$. In particular, as $\Gamma \cdot C=0$, the line bundle $\O_S(C)$ is not ample. 
\end{remark}

\begin{example} \label{ex:genere7}
  We give an example of a genus $7$ curve $C$ on a $K3$ surface such that $h^0(\N_{C/\PP^{6}}(-2))=2$. By standard techniques using lattice theory and the surjectivity of the period map, one can prove the existence of a $K3$ surface $S$ carrying a smooth irreducible elliptic curve $E$ and three smooth irreducible rational curves
$\Gamma$, $\Gamma_1$ and $\Gamma_2$ such that $E \cdot \Gamma=\Gamma \cdot \Gamma_1=\Gamma_1 \cdot \Gamma_2=1$ and $E \cdot \Gamma_1=E  \cdot \Gamma_2=\Gamma \cdot \Gamma_2=0$, cf. \cite[Fourth row of table p.~145]{JK}. Then $\O_C(E)$ induces a linear system of type $g^1_3$ on any smooth $C \in |4E+3\Gamma+2\Gamma_1+\Gamma_2|$ and $\omega_C \cong \O_C(4E)$, as $C \cdot \Gamma= C \cdot \Gamma_1= C \cdot \Gamma_2=0$. Thus, $h^0(\omega_C-3\O_C(E)) =h^0(\O_C(E))=2$, so that $h^0(\N_{C/\PP^{6}}(-2))=2$, by Proposition \ref{prop:c=1}(iv). One can check that
$|C|$ defines a birational morphism contracting $\Gamma$, $\Gamma_1$ and $\Gamma_2$, thus the projective model of $S$ has an $A_3$-singularity. 

In a similar way, one can construct examples of curves of genera $8$, $9$ and $10$ on $K3$ surfaces with
$h^0(\N_{C/\PP^{g-1}}(-2))=1$. We list the cases, which occur in \cite{JK}. In all cases, $E$ is a smooth elliptic curve and $\Gamma$ and $\Gamma'$ are smooth rational curves. The projective models have an $A_1$-singularity coming from the contraction of $\Gamma$. Propositions \ref{prop:main2} and \ref{prop:main3} imply that all cases with
$C$ trigonal, $ 8 \leq g \leq 10$ and $h^0(\N_{C/\PP^{g-1}}(-2))>0$ occur on {\it singular} projective models of $K3$ surfaces.
\vspace{0.2cm}

\begin{tabular}{ |l | l |l|l|}
  \hline			
  $g$ & $C\sim$ & intersections & appearance in \cite{JK} \\\hline\hline
$8$ & $4E+2\Gamma+\Gamma'$ & $E \cdot \Gamma=E \cdot \Gamma'=1$, $\Gamma \cdot \Gamma'=0$ & third row table p.~146 \\
$9$ & $5E+3\Gamma+\Gamma'$ & $E \cdot \Gamma=\Gamma \cdot \Gamma'=1$, $E \cdot \Gamma'=0$ & fourth row table p.~148  \\
$10$ & $6E+3\Gamma$ & $E \cdot \Gamma=1$ & fifth row table on p.~151 \\
  \hline  
\end{tabular}

\vspace{0.2cm}

Hence all the maximal values of $h^0(\N_{C/\PP^{6}}(-2))$ in Proposition \ref{prop:c=1} actually occur on curves on $K3$ surfaces. At the same time, also the remaining values occur for curves on $K3$ surfaces, as a consequence of Proposition \ref{prop:main2}. 
\end{example}

Assume now that $S \subset \PP^g$ is a smooth $K3$ surface of degree $2g-2$, with $g \geq 5$, all of whose hyperplane sections have Clifford index one, and set $H:=\O_S(1)$. By the classical results of Saint-Donat \cite{SD} and the fact that $H$ is ample, all smooth hyperplane sections are trigonal and the $g^1_3$ is induced by an elliptic pencil $|E|$ on the surface satisfying $E \cdot H=3$ (see, e.g., \cite[Thm. 1.3]{JK} for the precise statement). It is proved in \cite[\S 5]{JK} that one can find a pencil such that the three-dimensional rational normal scroll $T \subset \PP^g$ swept out by the span of the members of $|E|$ in $\PP^g$ (which are  plane cubics) is smooth (of degree $g-2$) and furthermore such that
  \begin{equation}
    \label{eq:h1R'}
   h^1(H-E)=h^1(H-2E)=0 
  \end{equation}
(the first vanishing by \cite[(2.6)]{Ma} or \cite[Prop. 2.6]{JK} and the latter by \cite[Prop. 5.5]{JK}, noting that the exceptional cases labeled (E0)-(E4) in \cite[Prop. 5.5]{JK} do not occur for ample $H$). Moreover, by \cite[Prop. 7.2]{JK}, the surface $S \subset \PP^g$ is cut out in $T$ by a section of $\O_T(3)(-(g-4)\F)$, where $\F$ is the class of the ruling of $T$, and the scroll type of $T$ is $(e_1,e_2,e_3)$ with $e_1+e_2+e_3=g-2$. (For the notion of scroll type and how to calculate it, cf., e.g., \cite{sc,Br,ste,JK}.) The possible scroll types occuring have been studied in \cite[2.11]{Reid0}, \cite[(1.7)]{ste} and \cite[\S 9.1]{JK}. 

\begin{lemma} \label{lemma:c=11}
  Let $C \subset S$ be a general hyperplane section. We have 
\[ h^0(\N_{S/\PP^g}(-2))=h^0(\N_{C/\PP^{g-1}}(-2))=0, \]
 except precisely in the following cases:
  \begin{itemize}
  \item[(i)] If $g=5$, then $h^0(\N_{S/\PP^5}(-2))=h^0(\N_{C/\PP^{4}}(-2))=3$;
   \item[(ii)] If $g=6$, then $h^0(\N_{S/\PP^6}(-2))=1$ and $h^0(\N_{C/\PP^{5}}(-2))=2$;
\item[(iii)] If $g=7$, then $h^0(\N_{C/\PP^{6}}(-2))=1$ and
\[ h^0(\N_{S/\PP^7}(-2))=\begin{cases} 1, & \mbox{if} \; H \sim 3E+\Gamma_1+\Gamma_2+\Gamma_3, \\ & \mbox{where} \; \Gamma_1,\Gamma_2,\Gamma_3 \; \mbox{are disjoint lines;} \\
0, & \mbox{otherwise}.\end{cases}
\]
\end{itemize}
\end{lemma}

\begin{remark} \label{c=1,g=7}
  As will be seen in the proof below, the two cases for $g=7$ occur, respectively, when $h^0(H-3E)=1$ and $0$. Moreover, the scroll type of $T$ is, respectively, $(3,1,1)$ and $(2,2,1)$ (cf. also \cite[\S 9.1 and table on p.~144-145]{JK}).
\end{remark}

\renewcommand{\proofname}{Proof of Lemma \ref{lemma:c=11}}

\begin{proof}
  The normal bundle sequence twisted by $-2$ yields
\[
\xymatrix{ 0 \ar[r] & \N_{S/T}(-2) \ar@{=}[d]^{\wr} \ar[r] & \N_{S/\PP^{g}}(-2) \ar[r] & \N_{T/\PP^{g}}|_S(-2) \ar[r] & 0 \\
& \O_S(H-(g-4)E)  & & & 
}\]
Since $T$ satisfies property $N_2$  (this can for instance be seen using Green's hyperplane section theorem on syzygies \cite[Thm. 3.b.7]{gr}, since its general curve section is a rational normal curve, which, as is well-known, satisfies property $N_2$), we have $h^0(\N_{T/\PP^{g}}|_S(-2))=0$ by Lemma \ref{lemma:well-known2}. Hence
\begin{equation}
  \label{eq:trig}
  h^0(\N_{S/\PP^{g}}(-2))=h^0(\O_S(H-(g-4)E).
\end{equation}
Since $H \cdot (H-(g-4)E)=10-g$, we have $h^0(\N_{S/\PP^{g}}(-2))=0$ for
$g \geq 10$. The possible values of $h^0(\O_S(H-(g-4)E)$, together with the scroll types, have been found in \cite[\S 9.1]{JK}, but we repeat the arguments for the sake of the reader. 

We consider first the case $g=9$. By \eqref{eq:h1R'} and Riemann-Roch, one finds
\begin{equation}
  \label{eq:19-1}
  h^0(H)=10, \; \; h^0(H-E)=7, \; \; h^0(H-2E)=4.
\end{equation}
We have $(H-3E)^2=-2$ and $(H-3E) \cdot H =7$. Hence $h^0(H-3E) \geq 1$ by Riemann-Roch and Serre duality. We claim that
\begin{equation}
  \label{eq:19-2}
  h^0(H-3E) =1 \; \mbox{or} \; 2.
\end{equation}
Indeed, if by contradiction $h^0(H-3E) \geq 3$, write $|H-3E|=|M|+\Delta$, with $\Delta$ the fixed part and $h^0(M) \geq 3$. We have
\begin{eqnarray*}
 0 < \Delta \cdot H  & = & 3E \cdot \Delta+M \cdot \Delta + \Delta^2
= 3E \cdot \Delta+(M+\Delta)^2-M^2-M \cdot \Delta \\
 & = & 3E \cdot \Delta-2-M^2-M \cdot \Delta.
\end{eqnarray*}
Hence $E \cdot \Delta \geq 2$ if $M^2 >0$, and from $3=E \cdot (H-3E)=E \cdot M + E \cdot \Delta$, we obtain $E \cdot M \leq 1$, which is impossible. Therefore, $M^2=0$, which means that $M \sim lF$ for an elliptic pencil $|F|$ and $l \geq 2$.
Since $F \cdot H \geq 3$, we must have $M \sim 2F$, $F \cdot H=3$ and $\Delta \cdot H=1$. Hence $\Delta$ is a line, so that $\Delta^2=-2$. As $3-3E \cdot F =F \cdot (H-3E)=F \cdot \Delta \geq 0$, we must have 
$E \cdot F \leq 1$, whence $E \sim F$. It follows that $H \sim 5E + \Delta$, which implies $\Delta^2=-14$, a contradiction. This proves \eqref{eq:19-2}. We next claim that 
\begin{equation}
  \label{eq:19-3}
  h^0(H-4E) =0.
\end{equation}
Indeed, assume the contrary and write $\Delta=H-4E$. Then
$\Delta^2=-8$, $H \cdot \Delta=4$ and $E \cdot \Delta=3$. The first and last of these equations imply that $\Delta$ must contain at least three irreducible curves in its support; moreover, at least one of them, call it $\Gamma$, must satisfy $\Gamma \cdot E \geq \Gamma \cdot H >0$. Then $\Gamma \cdot \Delta 
=\Gamma \cdot (H-4E) \leq -3$, whence
$\Delta -2\Gamma \geq 0$, which implies $\Gamma \cdot E=\Gamma \cdot H=1$. 
Since $(H-4E-2\Gamma)^2=-4$ and $H \cdot (H-4E-2\Gamma)=2$, we must
have $H-4E-2\Gamma \sim \Gamma_1+\Gamma_2$, with $\Gamma_1$ and $\Gamma_2$ disjoint lines. Since $E \cdot (\Gamma_1+\Gamma_2)=1$, we may assume $E \cdot \Gamma_1=1$. But then 
\[ -3=\Gamma_1 \cdot (H-4E) =
\Gamma_1 \cdot (2\Gamma+\Gamma_1+\Gamma_2)=2\Gamma \cdot \Gamma_1-2
,\]
 which is impossible. This proves \eqref{eq:19-3}.

From \eqref{eq:19-3} we obtain that $h^0(H-5E) =0$, whence $h^0(\N_{S/\PP^{9}}(-2))=0$ by \eqref{eq:trig}.
From \eqref{eq:19-1}-\eqref{eq:19-3} we obtain the two possible scroll types $(3,2,2)$ (occurring if $h^0(H-3E)=1$)
and $(3,3,1)$ (occurring if $h^0(H-3E)=2$) for $T$. By \cite[Thm. 2.4]{Br} a general hyperplane section $T'$ of $T$ is a rational normal scroll of type $(4,3)$ in both cases. This is the scroll swept out by the members of the $g^1_3$,  on a  general hyperplane section $C$ of $S$. Hence $h^0(\omega_C-5\O_C(E))=0$. It follows, by Proposition \ref{prop:c=1}(iii), that $h^0(\N_{C/\PP^{8}}(-2))=0$.

We next consider the case $g=8$. Similar considerations as in the previous case show that
\begin{eqnarray}
  \label{eq:18-1}
  h^0(H)=9, \; \; h^0(H-E)=6, \; \; h^0(H-2E)=3, \\
\nonumber h^0(H-3E)=0 \; \mbox{or} \; 1, \; \; h^0(H-4E)=0.
\end{eqnarray}
We prove only the last vanishing here. We have $(H-4E)^2=-10$ and $H \cdot (H-4E)=2$. The latter implies that, if effective, $H-4E$ is linearly equivalent to a sum of at most two rational curves, counted with multiplicity. Hence $(H-4E)^2 \geq -8$, again a contradiction.

In particular, \eqref{eq:18-1} implies that $h^0(\N_{S/\PP^{8}}(-2))=0$ by \eqref{eq:trig} and that the two possible scroll types of $T$ are $(2,2,2)$ 
(occurring if $h^0(H-3E)=0$) and $(3,2,1)$ (occurring if $h^0(H-3E)=1$). By \cite[Thm. 2.4]{Br} a general hyperplane section $T'$ of $T$ is a rational normal scroll of type $(3,3)$ in both cases. This is, as before, the scroll swept out by the members of the $g^1_3$  on a  general hyperplane section $C$ of $S$. Hence $h^0(\omega_C-4\O_C(E))=0$. It follows, by Proposition \ref{prop:c=1}(iii), that $h^0(\N_{C/\PP^{7}}(-2))=0$.

Assume now that $g=7$. Similar considerations as above show that
\begin{equation}
  \label{eq:17-1}
  h^0(H)=8, \; \; h^0(H-E)=5, \; \; h^0(H-2E)=2, \; \; h^0(H-3E)=0 \; \mbox{or} \; 1.
\end{equation}
More precisely, as $(H-3E)^2=-6$ and $H \cdot (H-3E)=3$, we have
\[
  h^0(H-3E)=1 \; \; \mbox{if and only if $H-3E$ is the sum of three disjoint lines.}
\]
In particular, \eqref{eq:17-1} yields the two possible scroll types $(2,2,1)$ 
(occurring if $h^0(H-3E)=0$) and $(3,1,1)$ (occurring if $h^0(H-3E)=1$), with
$h^0(\N_{S/\PP^{7}}(-2))=0$ and $1$, respectively, by \eqref{eq:trig}. Moreover, by \cite[Thm. 2.4]{Br} a general hyperplane section $T'$ of $T$ is a rational normal scroll of type $(3,2)$ in both cases, which yields $h^0(\omega_C-3\O_C(E))=1$. It follows, by Proposition \ref{prop:c=1}(iv), that $h^0(\N_{C/\PP^{6}}(-2))=1$ in both cases.

Assume that $g=6$. Then $h^0(\N_{C/\PP^{5}}(-2))=2$ by Proposition \ref{prop:c=1}(v). 
Since $(H-2E)^2=-2$ and $H \cdot (H-2E) >0$, we have $h^0(H-2E)=1$ by \eqref{eq:h1R'} and Riemann-Roch, whence $h^0(\N_{S/\PP^{6}}(-2))=1$ by 
\eqref{eq:trig}.

Assume that $g=5$. Then $h^0(\N_{C/\PP^{4}}(-2))=3$ by Proposition \ref{prop:c=1}(vi). 
We have $(H-E)^2=2$,
whence $h^0(H-E) =3$ by  \eqref{eq:h1R'} and Riemann-Roch, so that
$h^0(\N_{S/\PP^{5}}(-2))=3$ by \eqref{eq:trig}.
\end{proof}
\renewcommand{\proofname}{Proof}

\section{The case of Clifford index two} 

It is well-known that a smooth curve of Clifford index two is either tetragonal  (not isomorphic to a smooth plane quintic) or isomorphic to a smooth plane sextic, cf., e.g., \cite{elms}. If $g=5$, then $C$ is a complete intersection of three quadrics in $\PP^4$, whence $h^0(\N_{C/\PP^4}(-2))=3$. 

The following result is in principle proved in the unpublished preprint \cite{Te}, albeit with a small gap, cf. \cite[Rem. 2.17]{KL-gauss}. It is also present in \cite[Table 2, p.~161]{Te-duke}, referring to another unpublished preprint \cite{Te2} for a proof. The result can also be deduced from \cite[Thm. 5.6]{W2-coho}
or from 
\cite[Thm. 2.16 and Prop. 2.19]{ste}.
We give a proof following the arguments in the proof of \cite[Prop. 2.18]{KL-gauss}.

\begin{prop} \label{prop:c=2}
  Let $C$ be a smooth tetragonal curve of genus $g \geq 6$, not isomorphic to a smooth plane quintic. Then $h^0( \N_{C/\PP^{g-1}}(-2)) \leq 1$, with equality if and only if 
  \begin{itemize}
  \item[(i)] $C$ is bielliptic; or
  \item[(ii)] $7 \leq g \leq 9$ and $C$ in its canonical embedding is a quadratic section of either the anticanonical image of $\PP^2$ blown up in $10-g$ possibly infinitely near points or the $2$-Veronese embedding in $\PP^8$ of an irreducible quadric in $\PP^3$; or
\item[(iii)] $g=6$.
  \end{itemize}
 \end{prop}

\begin{proof}
 Let $|A|$ be any $g^1_4$ on $C$. By \cite[\S 6.2]{sc}, a hyperplane section curve $C \subset \PP^{g-1}$ lies in a rational normal scroll spanned by the divisors in $|A|$, not intersecting its possibly empty singular locus. In the desingularization of the scroll, denote by $H$ and $\R$ the pullbacks of the hyperplane bundle and ruling of the scroll, respectively. Then there are
two surfaces $\widetilde{Y}_A \sim 2H-b_{1,A}\R$ and $\widetilde{Z}_A \sim 2H-b_{2,A}\R$, for integers $b_{1,A} \geq b_{2,A} \geq 0$ such that $b_{1,A} + b_{2,A}=g-5$, with images $Y_A$ and $Z_A$ in $\PP^{g-1}$ so that $C = Y_A \cap Z_A$. 
Applying $\Shom_{\O_{\PP^{g-1}}}(-,\O_C)$ to the exact sequence
\[ 
\xymatrix{ 
0 \ar[r] & \I_{Y_A/\PP^{g-1}} \ar[r] & \I_{C/\PP^{g-1}} \ar[r] & \I_{C/Y_A} \ar[r] & 0}
\]
and twisting, we obtain
\[ 
\xymatrix{ 0  \ar[r] & \N_{C/Y_A}(-2) \ar@{=}[d]^{\wr} \ar[r] &  \N_{C/\PP^{g-1}}(-2)  \ar[r] &    \Shom_{\O_{\PP^{g-1}}}(\I_{Y_A/\PP^{g-1}},\O_C)(-2) \\
& \O_C(-b_{2,A}\R) & & 
}
\]
By \cite[Lemma 2.16]{KL-gauss}, the surface $Y_A \subset \PP^{g-1}$ satisfies property $N_2$. Thus, by Lemma \ref{lemma:well-known2}, we obtain $h^0( \N_{C/\PP^{g-1}}(-2))=h^0(\O_C(-b_{2,A}\R))$, that is, 
\begin{itemize}
\item $h^0( \N_{C/\PP^{g-1}}(-2))=0$ if $b_{2,A}>0$;
\item $h^0( \N_{C/\PP^{g-1}}(-2))=1$ if $b_{2,A}=0$.
\end{itemize}
In particular, the last case always happens when $g=6$. Moreover,  which is not a priori obvious, the invariant $b_{2,A}$ is either zero for all $A$, or nonzero for all $A$. 

Assume now that $g \geq 7$ and $b_{2,A}=0$. Then $Z_A$ is a quadric and, by 
\cite[Lemma 2.16]{KL-gauss}, $Y_A \subset \PP^{g-1}$ has degree $g-1$ and is linearly normal. Such surfaces are classified by \cite[Thm. 8]{Na}: 
\begin{itemize}
\item[(a)] $Y_A$ is the anticanonical image of $\PP^2$ blown up $10-g$ possibly infinitely near points;
\item[(b)]  $Y_A$ is the $2$-Veronese embedding in $\PP^8$ of an irreducible quadric in $\PP^3$;
 \item[(c)]  $Y_A$ is the cone over a smooth elliptic curve in $\PP^{g-2}$;
\item[(d)]  $Y_A$ is the $3$-Veronese embedding in $\PP^9$ of $\PP^2$.  
\end{itemize}

In case (d) the curve $C$ is isomorphic to a smooth plane sextic, whence pentagonal, a contradiction.

This leaves us with cases (a)-(c).
In case (c) the curve $C$ is bielliptic, and in cases (a)-(b) we have $g \leq 9$.

Conversely, if $C$ is bielliptic, or a quadratic section of a surface as in (a)-(b), then $b_2=0$ by \cite[Prop. 3.2]{Br} (see also \cite{sc}).
 \end{proof}

 \begin{remark} \label{rem:g=9}
When $g=9$, the cases where $C$ is a quadratic section of the anticanonical model of the Hirzebruch surface $\FF_1$ and of the $2$-Veronese embedding of a quadric in $\PP^3$ occur, respectively, when $C$ possesses a $g^2_6$ and a $g^3_8$, see \cite[(6.2)]{sc} or \cite[(3.2)]{sagra}. The two cases are mutually exclusive, since in the second case the two rulings on $\PP^1 \x \PP^1$ induce two $g^1_4$s on $C$, whereas in the first case $C$ possesses a unique $g^1_4$ (the $g^2_6$ maps the curve to a plane sextic with one singular point for reasons of genus and the unique $g^1_4$ is known to be cut out by lines through the singular point; alternatively, the curve can be embedded in $\FF_1$ and one may use
\cite[Cor. 1]{mar-hir}).
 \end{remark}

The next result must be well-known to the experts. We give a proof for lack of a reference, following the proof of  \cite[Prop. 2.9(d)]{KL-gauss} for smooth plane quintics.

\begin{prop} \label{prop:c=2SPS}
  If $C$ is isomorphic to a smooth plane sextic (whence of genus $10$), then 
$h^0( \N_{C/\PP^{g-1}}(-2)) =1$.
\end{prop}

\begin{proof}
Let $A$ be the very ample line bundle giving the embedding in $\PP^2$. As $\omega_C \cong \O_C(3A)$, in the canonical embedding $C$ is contained in the $3$-Veronese surface $Y \subset \PP^9$. We have a short exact sequence
\[
\xymatrix{ 0 \ar[r] & \N_{C/Y}(-2) \ar[r] & \N_{C/\PP^9}(-2)  \ar[r] &  \N_{Y/\PP^9}|_C (-2) \ar[r] & 0}
\] 
Since $Y$ satisfies condition $N_2$, we have $h^0( \N_{Y/\PP^9}|_C (-2))=0$ by 
Lemma \ref{lemma:well-known2}. Then the result follows since
$\N_{C/Y}(-2) \cong \O_C(6A-2 \cdot 3A)\cong \O_C$.
\end{proof}

We will now turn to curves on $K3$ surfaces and
start with the following famous example, the only occurrence of variation of gonality among smooth curves in a complete linear system on a $K3$ surface, by \cite{cp,Kn2}, and the only occurrence of smooth plane sextics, by \cite[Thm. 1.2]{Kn2}:

\begin{example}{\rm (Donagi-Morrison \cite[(2.2)]{DM})}. \label{ex:DM} Let $\pi: S \to \PP^2$ be a $K3$
surface of genus $2$, i.e. a double cover of $\PP^2$ branched along a smooth sextic, and
let $L := \pi^*\O_{\PP^2}(3)$. The arithmetic genus of the curves in $|L|$ is 
$10$. The smooth
curves in the codimension one linear subspace $\pi^*|H^0(\O_{\PP^2} (3))| \subset |L|$ are biellliptic,
whence of gonality $4$, whereas the general curve in $|L|$ is isomorphic
to a smooth plane sextic and therefore has gonality $5$.

The embedded surface $S \subset \PP^{10}$ is a complete intersection of the cone $V$ over the $3$-Veronese surface in $\PP^9$ and a quadric $Q$.  Since $V$ is smooth along $S$, the normal bundle sequence twisted by $-2$ yields
\[
\xymatrix{ 0 \ar[r] & \N_{S/V}(-2) \cong \O_S  \ar[r] & \N_{S/\PP^{10}}(-2) \ar[r] & \N_{V/\PP^{10}}|_S(-2) \ar[r] & 0}.\]
Since $V$ satisfies property $N_2$ (as its general hyperplane section does), we have  $h^0(\N_{V/\PP^{10}}|_S(-2))=0$ by Lemma \ref{lemma:well-known2}. Therefore, 
$h^0(\N_{S/\PP^{10}}(-2))=1$. Similarly, $h^0(\N_{C/\PP^{9}}(-2))=1$. 
\end{example}

We recall the following well-known fact:

\begin{lemma} \label{lemma:desc2}
  Let $|H|$ be a complete linear system of curves of genus $g \geq 7$ on a $K3$ surface
  such that all smooth curves in $|H|$ have Clifford index two. 

 Then, except for the Donagi-Morrison example \ref{ex:DM}, all smooth curves in 
$|H|$ are tetragonal, and, for any line bundle $A$ of type $g^1_4$ on any smooth $C \in |H|$, there is a globally generated line bundle $\O_S(D)$ on $S$ satifying $\O_C(D) \geq A$, 
$h^i(D)=h^i(H-D)=0$, $i=1,2$, and one of the three conditions
\begin{itemize}
\item $D^2=0$, $D \cdot H=4$; 
\item $D^2=2$, $D \cdot H=6$; $7 \leq g \leq 9$;
 \item $D^2=4$, $H \sim 2D$, $g=9$.
\end{itemize}
Moreover, $\O_C(D)$ computes the Clifford index of any smooth $C \in |H|$.
\end{lemma}

\begin{proof}
The fact that all smooth  curves in 
$|H|$ are tetragonal except for the Donagi-Morrison example follows from \cite[Thm. 1.2]{Kn2}. Moreover,
by nowadays well-known Reider-like results as in \cite{DM,cp,Kn}, for any $C$ and $A$ as in the statement, there is a line bundle $\O_S(D)$ on $S$ such that $\O_C(D) \geq A$ and satisfying $0 \leq D^2\leq 4$ and $D \cdot H=D^2+4$ (see, e.g., \cite[Lemma 8.3]{Kn}). The Hodge index theorem yields the cases stated in the lemma, in addition to the possibility $D^2=2$ and $H \sim 3D$, which is the Donagi-Morrison example \ref{ex:DM}. The fact that $\O_S(D)$ can be chosen globally generated and such that $h^i(D)=h^i(H-D)=0$, $i=1,2$, follows from \cite[(2.3)]{Ma} (see also
\cite[Proof of Lemma 8.3]{Kn} and \cite[Props. 2.6 and 2.7]{JK}). The fact that 
$\O_C(D)$ computes the Clifford index of any smooth $C \in |H|$ is standard and easily checked.
\end{proof}

\begin{remark} \label{rem:ellpenc}
  When $D^2=0$, then $|D|$ is an elliptic pencil and $\O_C(D)=A$. 
\end{remark}

\begin{corollary} \label{cor:desc2}
  Let $C$ be a smooth curve of genus $g \geq 10$ lying on a $K3$ surface, such that $\O_S(C)$ is not as in the Donagi-Morrison example \ref{ex:DM}. Then $C$ is neither bielliptic nor isomorphic to a smooth plane sextic.
\end{corollary}

\begin{proof}
 If $C$ is bielliptic, then $C$ must contain infinitely many $g^1_4$s, which is impossible when $g \geq 10$ by Lemma \ref{lemma:desc2} and Remark \ref{rem:ellpenc} if we are not in the Donagi-Morrison example \ref{ex:DM}, recalling that there are finitely many elliptic pencils of degree $4$ with respect to $C$ on $S$.
The fact that smooth plane sextics only occur in the  Donagi-Morrison example follows from \cite[Thm. 1.2]{Kn2}.
\end{proof}

\begin{remark} \label{rem:biell}
(a) The fact that there are no bielliptic curves of genus $g \geq 11$ on a $K3$ surface is a well-known result of Reid's \cite[Cor. 2]{Reid}, already mentioned in the introduction.

(b)
By \cite[Thm. 3.12]{AF} and \cite[Thm. 1.2]{Kn2}, 
there exists no smooth bielliptic curve of genus $g$ with $6 \leq g \leq 9$ on a $K3$ surface that is general in its complete linear system. (Indeed, as $\rho(g,1,4) \leq 0$, by \cite[Thm. 3.12]{AF} any curve of Clifford dimension one and general in its linear system on a $K3$ has a finite number of pencils computing its gonality. Then the result follows as curves of of genus $\leq 9$ and Clifford dimension $>1$ lying on $K3$ surfaces  are only smooth plane quintics by \cite[Thm. 1.2]{Kn2}, which cannot be bielliptic, cf., e.g., \cite[\S 2.2]{CM}.)
Thus, case (i) in Proposition \ref{prop:c=2} never occurs if $C$ is general in its linear system on a $K3$ surface.
\end{remark}

Assume now (and for the rest of the section)  that $S \subset \PP^g$ is a smooth $K3$ surface, all of whose hyperplane sections have Clifford index two. By \eqref{eq:SCk2}, \eqref{eq:Sk2}, Proposition \ref{prop:c=2} and Corollary \ref{cor:desc2}, the cases for which $h^0(\N_{S/\PP^g}(-2)) \neq 0$ apart from the Donagi-Morrison example \ref{ex:DM} must satisfy $g \leq 9$.

The next example is well-known:

\begin{example} \label{ex:old}
  If $g=5$, then $S \subset \PP^5$ is a complete intersection of three quadrics, so that $\N_{S/\PP^5}\cong  \O_S(2)^{\+3}$ and $h^0(\N_{S/\PP^5}(-2))=h^0(\N_{C/\PP^4}(-2))=3$. 

If $g=6$, then $S$ is $BN$ general in the sense of Mukai (cf., e.g., \cite[Prop. 10.5]{JK}), so that by \cite{Mu}, $S$ is a quadratic  section of a (possibly singular) quintic Del Pezzo threefold $V$ in $\PP^6$ (in turn a hyperplane section of a 
quintic Del Pezzo fourfold in $\PP^7$). As in Example \ref{ex:DM}, one proves that $h^0(\N_{S/\PP^6}(-2))=h^0(\N_{C/\PP^5}(-2))=1$.
\end{example}

The next four lemmas will be the necessary ingredients to finish the proof of Proposition \ref{prop:main2} in the next section.

\begin{lemma} \label{lemma:doppio}
  We have $h^0(\N_{S/\PP^g}(-2))= h^0(\N_{C/\PP^{g-1}}(-2))=1$ in the following cases:
  \begin{itemize}
    \item[(i)] $g=9$ and $H \sim 2D$ with $D^2=4$. A general such $S$ is the $2$-Veronese embedding of a quartic in $\PP^3$, and thus a quadratic section of the $2$-Veronese embedding of $\PP^3$ in $\PP^{9}$. Conversely, any such smooth quadratic section is a $K3$ surface carrying such a divisor $D$.
    \item[(ii)] $g=7$ (resp., $8$) and there is a globally generated line bundle $D$ on $S$ satisfying $D^2=2$ and $D \cdot H=6$. A general such $S$ is a quadratic section of the sextic Del Pezzo threefold $W$ in $\PP^7$ that is a divisor of bidegree $(1,1)$ in $\PP^2 \x \PP^2$ (resp., a quadratic section of a blow up of $\PP^3$ at a point embedded in $\PP^8$ by the linear system of quadrics through the point). Conversely, any such smooth quadratic section is a $K3$ surface carrying such a divisor $D$.
\end{itemize}
\end{lemma}

\begin{proof}
  (i) Since $h^0(\N_{S/\PP^9}(-2)) \leq h^0(\N_{C/\PP^{g-1}}(-2)) \leq 1$ by Proposition \ref{prop:c=2}, we may assume that the pair $(S,D)$ is general in moduli, in particular that $D$ is very ample. (By the classical results of \cite{SD}, $D$ is not very ample if and only if there is a smooth rational curve $\Gamma$ such that $\Gamma \cdot D=0$ or a smooth elliptic curve $F$ such that $F \cdot D \leq 2$.) Then $S$ is the $2$-Veronese embedding of a quartic in $\PP^3$, and thus a quadratic section of the $2$-Veronese embedding of $\PP^3$ in $\PP^{9}$. As in Example \ref{ex:DM}, one computes $h^0(\N_{S/\PP^9}(-2))=1$.

(ii) As above, we may assume that $(S,H,D)$ is general in moduli, in particular that $F:=H-D$ is base point free and defines a birational map.

If $g=8$, set $\Delta:=H-2D$. Then $\Delta^2=-2$ and $H \cdot \Delta=2$, whence $\Delta$ is effective (an irreducible conic by generality) by Riemann-Roch. The complete linear system $|F|$ is base point free and maps $S$ birationally onto a quartic surface in $\PP^3$, contracting $\Delta$ to a point. Let $\pi: \widetilde{\PP^3} \to \PP^3$ be the blow up at this point, and let $\mathfrak{E}$ be the exceptional divisor. Then $S \in |2(\pi^*\O_{\PP^3}(2)-\mathfrak{E})|$ and is thus a quadratic section of $\widetilde{\PP^3}$ embedded into $\PP^8$ by $|\pi^*\O_{\PP^3}(2)-\mathfrak{E}|$, which restricted to $S$ becomes $|2(H-D)-\Delta|=|H|$, as claimed. As in Example \ref{ex:DM}, one computes $h^0(\N_{S/\PP^8}(-2))=1$. Conversely, it is easily checked that any such smooth quadratic section has the desired properties. 

If $g=7$, the linear systems $|D|$ and $|F|$ define an embedding
\[ S \subset \PP^2 \x \PP^2 \subset \PP^8,\]
where the right hand embedding is the Pl{\"u}cker embedding, which factors through the embedding $S \subset \PP^7$ defined by $|H|$. Thus
$S \subset \left(\PP^2 \x \PP^2\right) \cap \PP^7 \subset \PP^8$. A priori, the intersection $T:= \left(\PP^2 \x \PP^2\right) \cap \PP^7$ does not need to be transversal. However, assuming first it is, $T$ is a sextic Del Pezzo threefold, with $\omega_T \cong \O_T(-2)$, so that a smooth quadratic section of $T$ will be a 
$K3$ surface with the desired properties. As we assume that $(S,H,D)$ is general, we can thus assume that $S$ is a quadratic section of the sextic Del Pezzo threefold $T$ (see also \cite[Lemma 4.1]{KLV}). As in Example \ref{ex:DM}, one computes $h^0(\N_{S/\PP^7}(-2))=1$.
\end{proof}

\begin{lemma} \label{lemma:92}
  Assume that $g=9$ and there is a globally generated line bundle $D$ on $S$ satisfying $D^2=2$ and $D \cdot H=6$. Assume furthermore that $H$ is not $2$-divisible and that it is not of the form $H \sim 3E+2\Delta$, where $|E|$ is an elliptic pencil and $\Delta$ is an effective divisor such that $\Delta^2=-2$ and $\Delta \cdot E=2$.

Then $h^0(\N_{S/\PP^9}(-2))=0$ and $h^0(\N_{C/\PP^8}(-2))=1$ for any smooth $C \in |H|$.
\end{lemma}

\begin{proof}
We first prove that $h^0(\N_{C/\PP^8}(-2))=1$. If $C$ is bielliptic, there is nothing to prove by Proposition \ref{prop:c=2}. Otherwise, the base point free complete linear system $|D|$ maps $C$ birationally to a plane curve of degree $6$, whence with one ordinary singular point, for reasons of genus. Blowing up the point, we get an embedding of $C$ into $\FF_1$ linearly equivalent to twice the anticanonical bundle. Thus,  $h^0(\N_{C/\PP^8}(-2))=1$ again by Proposition \ref{prop:c=2}.

We next prove that $h^0(\N_{S/\PP^9}(-2))=0$. We will use Lemma \ref{lemma:trick} and prove that $h^1(\T_S(-2H))=0$. 
 Contrary to the previous proof, we cannot assume that $(S,H,D)$ is general in this case. 

Set $F:=H-D$. Then $F^2=6$ and $F \cdot D=4$. 
We first claim that $F$ is ample. 

To prove this, assume to get a contradiction, that there exists an irreducible curve $\Gamma$ such that $\Gamma \cdot F \leq 0$. Then $\Gamma^2=-2$. It is easy to check that $\O_C(F-\Gamma)$ contributes to the Clifford index of any smooth curve $C \in |H|$ and that
\begin{eqnarray*}
 \Cliff \O_C(F-\Gamma) & = & (F-\Gamma)\cdot H - 2h^0(\O_C(F-\Gamma))+2 \\
& \leq & 
(F-\Gamma)\cdot H - 2h^0(F-\Gamma)+2 \\
& \leq & (F-\Gamma)\cdot H -(F-\Gamma)^2-2 = 4-\Gamma \cdot H+2F \cdot \Gamma.
\end{eqnarray*}
 Thus, by the assumption that $\Cliff C=2$, we must have $F \cdot \Gamma=0$ and $\Gamma \cdot H=\Gamma \cdot D=1$ or $2$. If $\Gamma \cdot H=2$, then 
$(D+\Gamma)^2=4$ and $(D+\Gamma) \cdot H=8$, whence
the Hodge index theorem yields $H \sim 2(D+\Gamma)$, contrary to our assumptions. If
$\Gamma \cdot H=1$, then $G:=F-\Gamma-D$ satisfies $G^2=0$ and $G \cdot H=3$, so that $|G|$ cuts out a $g^1_3$ on all $C \in |H|$, again a contradiction.
Hence $F$ is ample. 

We next claim that $|F|$ is base point free. Indeed, if it is not, then by \cite[(2.7)]{SD}, we would have $F \sim 4E+\Gamma$, for an elliptic pencil $|E|$ and a smooth rational curve $\Gamma$ such that $\Gamma \cdot E=1$. But then, as $F \cdot H=10$, we would have $E \cdot H \leq 2$, so that all smooth curves in $|H|$ would be hyperelliptic, a contradiction. 

We finally claim that $F$ is very ample. Indeed, if it is not, then by \cite[Thm. 5.2]{SD}, there would exist an elliptic pencil $|E|$ such that $E \cdot F=2$.
Set $\Delta:=F-2E$. Then $\Delta^2=-2$ and $\Delta \cdot F=2$, whence $\Delta$ is effective by Riemann-Roch. Since the Clifford index of any smooth $C \in |H|$ is $2$, we must have $E \cdot H \geq 4$. 
From $10=F \cdot H= (2E + \Delta) \cdot H$, we thus find that $E \cdot H=4$ and
$\Delta \cdot H=2$. The Hodge index theorem  implies that $H \sim 3E+2\Delta$, contrary to our assumptions. 

Therefore, $|F|$ defines an embedding of $S$ into $\PP^4$ and its image is well-known to be a complete intersection of a quadric and a cubic. 
  The Euler sequence of the embedding $S \subset \PP^4$ twisted by $\O_S(-2H)$ is
\[
0 \longrightarrow  \O_S(-2H) \longrightarrow  H^0(F)^{\vee} \* \O_S(-F-2D) \longrightarrow   \T_{\PP^4}|_S(-2H)  \longrightarrow  0.\]
The map on cohomology $H^2(\O_S(-2H)) \to H^0(F)^{\vee} \* H^2(\O_S(-F-2D))$ is the dual of the multiplication map of sections 
\[ \mu_{F,F+2D}: H^0(F) \* H^0(\O_S(F+2D)) \to H^0(2H).\]
 Since $h^1(-2D)=0$ and $h^0(F-2D)=0$ (using $H \cdot (F-2D)=-2$), this map is surjective by Mumford's generalization of a theorem of Castelnuovo \cite[Thm. 2, p.~41]{Mum}. Thus $h^i(\T_{\PP^4}|_S(-2H))=0$ for $i=0,1$. From the exact sequence
\[
\xymatrix{ 0 \ar[r] & \T_S(-2H) \ar[r] & \T_{\PP^4}|_S(-2H)  \ar[r] & \N_{S/\PP^4}(-2H) \ar[r] & 0,}
\]
we therefore obtain that 
\begin{eqnarray*}
  h^1(\T_S(-2H))&  =&  h^0(\N_{S/\PP^4}(-2H)) =h^0(2F-2H)+h^0(3F-2H) \\
& = & h^0( -2D) +h^0(F-2D)=0. 
\end{eqnarray*}
It follows that $h^0(\N_{S/\PP^9}(-2))=0$ by Lemma \ref{lemma:trick}.
\end{proof}

\begin{remark} \label{rem:92}
  As seen in the proof, the condition on $H$ in Lemma \ref{lemma:92} can be rephrased as $H$ not being $2$-divisible and $H-D$ being very ample. 
\end{remark}

\begin{lemma} \label{lemma:rest2i}
  Assume that $7 \leq g \leq 9$ and that all $D \in \Pic S$ satisfying the   conditions  in Lemma \ref{lemma:desc2} satisfy $D^2=0$. Let $C \subset S$ be a general hyperplane   section. Then $h^0(\N_{S/\PP^g}(-2))=h^0(\N_{C/\PP^{g-1}}(-2))=0$ except in the following case where
$h^0(\N_{S/\PP^g}(-2))=h^0(\N_{C/\PP^{g-1}}(-2))=1$:

$g=7$ and $H \sim E_1+E_2+E_3$, where $|E_i|$ is an elliptic pencil, $i=1,2,3$, and $E_i \cdot E_j=2$ for $i \neq j$. A general such $S$ is a 
 quadratic section of the sextic Del Pezzo threefold $T \cong \PP^1 \x \PP^1 \x \PP^1$ in its Segre embedding in $\PP^7$; conversely, any such smooth quadratic section is a $K3$ surface satisfying the given properties.
\end{lemma}

\begin{proof}
Pick any  $D$ satisfying the conditions in Lemma \ref{lemma:desc2} and call it $E$. Then $E^2=0$ and $E \cdot H=4$ by assumption, and $|E|$ is an elliptic pencil, cf. Remark \ref{rem:ellpenc}. As in the case of Clifford index one, it is proved in \cite[\S 5]{JK} that one can find an $E$ such that the four-dimensional rational normal scroll $T \subset \PP^g$ swept out by the span of the members of $|E|$ in $\PP^g$ is smooth (of degree $g-3$), and furthermore such that
  \begin{equation}
    \label{eq:h1R}
   h^1(H-2E)=0,
  \end{equation}
the latter by \cite[Prop. 5.5]{JK}, noting that the exceptional cases labeled (E0)-(E4) in \cite[Prop. 5.5]{JK} do not occur for ample $H$. (Here the assumption about nonexistence of divisors $D$ as in Lemma \ref{lemma:desc2} with $D^2>0$ plays a central role, as we now briefly recall for the sake of the reader: if by contradiction $h^1(H-2E)>0$, then, as $(H-2E)^2 =2g-18 \geq -4$, we have $h^0(H-2E) =\chi(H-2E)+h^1(H-2E) \geq 1$, whence $H-2E$ is effective and not numerically $1$-connected. Therefore, we have a nontrivial effective decomposition $H-2E \sim A_1+A_2$ with $A_1 \cdot A_2 \leq 0$. One may check that $E+A_i$ for $i=1$ or $2$ satisfies the conditions in Lemma \ref{lemma:desc2} and $(E+A_i)^2>0$, a contradiction.) 
Moreover, by \cite[Prop. 7.2 and \S~9.2]{JK}, or \cite[\S 4]{Br} or \cite[\S 1.7]{ste}, the surface $S \subset \PP^g$ is 
a complete intersection of two threefolds
\begin{equation} \label{eq:ciinT}
 S = Y_1 \cap Y_2, \; \; \mbox{with} \; \; Y_i \in |\O_T(2)(-b_i\F)|, \; \; b_1 \geq b_2 \; \mbox{and} \; b_1+b_2=g-5,
\end{equation}
where, as before, $\F$ is the class of the ruling of $T$.

  The normal bundle sequence of $S \subset Y_1 \subset \PP^g$ twisted by $-2$ yields
\[
\xymatrix{ 0 \ar[r] & \N_{S/Y_1}(-2) \ar@{=}[d]^{\wr} \ar[r] & \N_{S/\PP^{g}}(-2) \ar[r] & \N_{Y_1/\PP^{g}}|_S(-2) \ar[r] & 0 \\
& \O_S(-b_2E)  & & & 
}\]
(using that $Y_1$ is smooth along $S$). Restricting to $C$ we obtain
\[
\xymatrix{ 0 \ar[r] & \O_C(-b_2E)  \ar[r] & \N_{C/\PP^{g-1}}(-2) \ar[r] & \N_{Y_1/\PP^{g}}|_C(-2) \ar[r] & 0}.\]
The threefold $Y_1$ satisfies property $N_2$ by Green's hyperplane section theorem \cite[Thm. 3.b.7]{gr}, since its general hyperplane section does by \cite[Lemma 2.16]{KL-gauss}. Thus, we have $h^0(\N_{Y_1/\PP^{g}}|_S(-2))=h^0(\N_{Y_1/\PP^{g}}|_C(-2))=0$ by 
Lemma \ref{lemma:well-known2}. Hence
\[ h^0(\N_{S/\PP^{g}}(-2))=h^0(\O_S(-b_2 E)) \; \; \mbox{and} \; \; h^0(\N_{C/\PP^{g-1}}(-2))=h^0(\O_C(-b_2 E)).\]
It follows that 
\begin{equation}
  \label{eq:tetra}
  h^0(\N_{S/\PP^{g}}(-2))=h^0(\N_{C/\PP^{g-1}}(-2))=\begin{cases} 0 & \mbox{if} \; b_2>0, \\
1 & \mbox{if} \; b_2=0.
\end{cases}
\end{equation}
 The possible values of $b_2$ (and $b_1$), and the possible scroll types $(e_1,e_2,e_3,e_4)$, with $e_1 \geq e_2 \geq e_3 \geq e_4 >0$ (as $T$ is smooth) have been investigated in \cite{Br,ste,JK}, with some minor mistakes in the former. Recall that $e_1+e_2+e_3+e_4=g-3$. We repeat the study of the case $g=9$ for the sake of the reader.

If $g=9$, we have $b_1+b_2=4$ and $e_1+e_2+e_3+e_4=6$. We may use Riemann-Roch to compute $h^0(H-E)=6$, as $h^1(H-E)=0$ by Lemma \ref{lemma:desc2}, and $h^0(H-2E)=2$, using \eqref{eq:h1R}. As $H \cdot (H-4E)=0$, we get $h^0(H-4E)=0$. We claim that $h^0(H-3E) \leq 1$. Indeed, if $h^0(H-3E) \geq 2$, write $|H-3E|=|M|+\Delta$, with $|M|$ the moving part and $\Delta$ the fixed part. 
Since
$(H-3E)^2=-8$ and $H \cdot (H-3E)=4$, we have $\Delta >0$ and $M \cdot H \leq 3$, so that $|M|$ would induce a $g^1_3$ on all curves in $|H|$, a contradiction. 
This yields the two possible scroll types $(2,2,1,1)$ and $(3,1,1,1)$ (cf. \cite[\S 9.2.2 and table on p.~148]{JK}). In the latter case, \cite[Lemma 8.33]{JK}, \cite[Lemma 1.9]{ste} or \cite[Prop. 5.4]{Br} yields $b_1=b_2=2$ (the reason being that any section of $\O_T(2)(-b\F)$ with $b \geq 3$ is a product of a section of $\O_T(1)(-b\F)$ and a section of $\O_T(1)$). In the former case, \cite[Lemma 1.9]{ste} or \cite[Prop. 5.4]{Br} (or the discussion in \cite[\S 9.2.2]{JK}) yields $b_1=2$ or $3$, whence $b_2 >0$ (the reason being that the zero scheme of any section of $\O_T(2)(-4\F)$ restricts to two lines in each fiber of $T$). In all cases we therefore have 
$h^0(\N_{S/\PP^{9}}(-2))=h^0(\N_{C/\PP^{8}}(-2))=0$ by \eqref{eq:tetra}.

If $g=8$, we have $b_1+b_2=3$ and similar considerations as in the previous case yield that the scroll type must be $(2,1,1,1)$ (cf. \cite[\S 9.2.2 and table on p.~146]{JK}). By \cite[Lemma 1.9]{ste} or \cite[Prop. 5.4]{Br} (or the discussion in \cite[\S 9.2.2]{JK}), we have $b_1=2$ and $b_2=1$. 
Hence $h^0(\N_{S/\PP^{8}}(-2))=h^0(\N_{C/\PP^{7}}(-2))=0$ by \eqref{eq:tetra}.
 
If $g=7$, we have $b_1+b_2=2$ and the scroll type must be $(1,1,1,1)$ (cf.  \cite[\S 9.2.2 and table on p.~144-145]{JK}). By \cite[Lemma 1.9]{ste}, \cite[Lemma 8.33]{JK} or \cite[Prop. 5.4]{Br} we must have $b_1 \leq 2$, whence the two possibilities $(b_1,b_2)=(1,1)$ or $(2,0)$. Both cases occur by  \cite[Thm. 5.3]{Br}, with $h^0(\N_{S/\PP^{8}}(-2))=h^0(\N_{C/\PP^{7}}(-2))=0$ and $1$, respectively, by \eqref{eq:tetra}. Let us now consider the second case more thoroughly. The general curves $C \in |H|$ are contained in hyperplane sections of the threefold $Y_2 \in |\O_T(2)|$, which are the surfaces $Y_A \subset \PP^6$ appearing in the proof of Proposition \ref{prop:c=2}. The only possibility is that $Y_A$ is the blow up of $\PP^2$ in three (possibly infinitely near) points, cf. Remark \ref{rem:biell}(b). Hence $C$ has precisely three linear systems of type $g^1_4$ by \cite[Prop. 3.4(d)]{KL-gauss}. As we are assuming that the only line bundles $D$ satisfying the conditions in Lemma \ref{lemma:desc2} are the ones with square zero, then there must exist three elliptic pencils $|E_i|$, $i=1,2,3$, with $E_1=E$, say,  on $S$ inducing these three linear systems on $C$; in particular, $E_i \cdot H=3$. An easy application of the Hodge index theorem yields $E_i \cdot E_j \leq 2$ for $i \neq j$, and clearly equality must hold, as otherwise we would have moving linear systems of degree one on an elliptic curve. It is an easy exercise to check that
$H \sim E_1+E_2+E_3$ and to check the remaining assertions. 
\end{proof}

\begin{lemma} \label{lemma:rest2ii}
  Assume that $g=9$ and $H \sim 3E+2\Delta$, where $|E|$ is an elliptic pencil and $\Delta$ is an effective divisor such that $\Delta^2=-2$ and $\Delta \cdot E=2$. Then $h^0(\N_{S/\PP^g}(-2))=h^0(\N_{C/\PP^{g-1}}(-2))=1$. 

Moreover, a general such $S$ is a 
 quadratic section of the cone over the anticanonical embedding of $\FF_1$ into
$\PP^8$; conversely, any such smooth quadratic section is a $K3$ surface satisfying the given properties.
\end{lemma}

\begin{proof}
 Both $E$ and 
$E+\Delta$ satisfy the conditions for the divisor $D$ in Lemma \ref{lemma:desc2}.
Contrary to the previous proof, the four-dimensional rational normal scroll $T_0 \subset \PP^g$ defined by the span of the members of $|E|$ in $\PP^g$ is singular. Indeed, one easily computes 
\[ h^0(H-E)=6, \; \; h^0(H-2E)=3, \; \; h^0(H-3E)=1, \; \; h^0(H-4E)=0,\]
so the resulting scroll type is $(3,2,1,0)$. A general hyperplane section of it
is a rational normal scroll $X$ of type $(3,2,1)$, by \cite[Thm. 2.4]{Br}, whence  $T_0$ is a cone over $X$. The scroll $X$ is defined by the spans of the members of the induced $g^1_4$ on a general hyperplane section $C$ of $S$. Denote by $T:=\PP(\E) \to T_0$, with $\E:=\O_{\PP^1}(3) \+\O_{\PP^1}(2)\+\O_{\PP^1}(1)\+\O_{\PP^1}$ the desingularization of the scroll $T_0$. Note that $S$ does not intersect the vertex of $T_0$, as $|E|$ is base point free, and that \eqref{eq:ciinT} and \eqref{eq:tetra} still hold (with $g=9$). Restricting to $X$, we get the two surfaces in $|\O_X(2)(-b_1\F)|$ and $|\O_X(2)(-b_2\F)|$ of which $C$ is a complete intersection, as in the proof of Proposition \ref{prop:c=2}. We now observe that the curve $C$ lies on $\FF_1$ linearly equivalent to twice the anticanonical section; indeed $|E+\Delta|$ defines a $g^2_6$ on $C$, mapping it to a plane sextic curve with one ordinary singular point, for reasons of genus; blowing up the plane in the singular point yields the desired embedding. Hence, by \cite[(6.2)]{sc}, we must have $b_2=0$, whence $b_1=4$. Thus, $S$ is a quadratic section of a threefold $Y_1 \in |\O_T(2)(-4\F)|$, and $C$ is a quadratic section of a hyperplane section of $Y_1$, which is the anticanonical embedding of $\FF_1$, as $C$ carries a $g^2_6$, cf. Remark \ref{rem:g=9}.  As the latter is well-known to be nonextendable,  $Y_1$ must be the cone over it. (Alternatively, one may check as in \cite{ste} that the base locus of $|\O_T(2)(-4\F)|$ is the inverse image of the vertex of $T_0$.) 
The rest of (ii) is easily verified.
\end{proof}

\section{Proof of Proposition \ref{prop:main2} and final remarks} \label{sec:K3}

The results in the two previous section are enough to deduce Proposition \ref{prop:main2}. We summarize for the sake of the reader.

\renewcommand{\proofname}{Proof of Proposition \ref{prop:main2}}

\begin{proof}
  Let $S \subset \PP^g$ be a smooth  $K3$ surface of degree $2g-2$, with $g \geq 5$. Let $c$ be the Clifford index of all smooth hyperplane sections of $S$, recalling that it is constant by \cite{gl}, and that $c >0$ by \cite{SD}. 

If $g=5$ or $6$, then $c=1$ or $2$, and the result follows from Lemma \ref{lemma:c=11}(i-ii) if $c=1$ and Example \ref{ex:old} if $c=2$.

Assume now that $g \geq 7$ and that $h^0(\N_{S/\PP^g}(-2)) \neq 0$. By \eqref{eq:Sk2} and Corollary \ref{cor:well-knownK3}, we have $g \leq 10$ and $c=1$ or $2$.

If $c=1$, then Lemma \ref{lemma:c=11} yields case (I) as the only case where one has  $h^0(\N_{S/\PP^g}(-2)) \neq 0$. 

We are therefore left with the cases $c=2$ and $7 \leq g \leq 10$.
The Donagi-Morrison example \ref{ex:DM} is case (VII), so we may henceforth assume that we are not in this case. Let $D$ be a divisor satisfying the conditions in Lemma \ref{lemma:desc2}. If $\O_S(1) \sim 2D$ (whence $g=9$ and $D^2=4$), Lemma \ref{lemma:doppio}(i) yields case (V). We may henceforth assume that $\O_S(1)$ is not $2$-divisible and that $D^2=0$ or $2$, with the latter implying $g \leq 9$.

If $D^2=2$ and $g=7$ or $8$, then Lemma \ref{lemma:doppio}(ii) yields cases (III) and (IV). If $D^2=2$ and $g=9$, then Lemmas \ref{lemma:92} and \ref{lemma:rest2ii}  yield case (VI) as the only case where  $h^0(\N_{S/\PP^g}(-2)) \neq 0$. 

Finally, assume that the only divisors $D$ as in Lemma \ref{lemma:desc2} satisfy $D^2=0$. Then Lemma \ref{lemma:rest2i} yields case (II) as the only case where  one has $h^0(\N_{S/\PP^g}(-2)) \neq 0$. 
\end{proof}

\renewcommand{\proofname}{Proof}

A thorough look at the cases above show that we have also classified the cases where $h^0(\N_{S/\PP^g}(-2)) \neq h^0(\N_{C/\PP^{g-1}}(-2))$:

\begin{proposition} \label{prop:main3} 
  Let $S \subset \PP^g$ be a smooth  $K3$ surface of degree $2g-2$, with $g \geq 5$, all of whose hyperplane sections have Clifford index $c$. Let $C \subset S$ be a general hyperplane section. Then $h^0(\N_{S/\PP^g}(-2))=h^0(\N_{C/\PP^{g-1}}(-2))$ except precisely in the following cases:
  \begin{itemize}
  \item[(a)]  $(g,c)=(6,1)$. Then $h^0(\N_{S/\PP^6}(-2))=1$ and 
$h^0(\N_{C/\PP^{5}}(-2))=2$;
   \item[(b)] $(g,c)=(7,1)$, and $H$ is not linearly equivalent to $3E+\Gamma_1+\Gamma_2+\Gamma_3$, where $|E|$ is an elliptic pencil and $\Gamma_1,\Gamma_2,\Gamma_3$ are disjoint lines. Then $h^0(\N_{S/\PP^7}(-2))=0$ and 
$h^0(\N_{C/\PP^{7}}(-2))=1$; 
\item[(c)] $(g,c)=(9,2)$, and $H$ is as in Lemma \ref{lemma:92}.
  \end{itemize}
\end{proposition}

\begin{proof}
  Among all the cases we have considered, the only ones where we have $h^0(\N_{S/\PP^g}(-2)) \neq h^0(\N_{C/\PP^{g-1}}(-2))$, are the ones in Lemma \ref{lemma:c=11}(ii-iii), yielding cases (a) and (b), and the one in Lemma \ref{lemma:92}.
\end{proof}

%
%

\end{document}